\documentclass[12pt,reqno]{amsart}
\usepackage[latin1]{inputenc}
\usepackage{amsmath}
\usepackage{amsfonts}
\usepackage{amssymb}
\usepackage{graphicx}
\usepackage{enumitem}
\usepackage{xcolor}
\usepackage{hyperref}
\usepackage{url}
\usepackage{amsthm}
\usepackage{bm}
\newtheorem{definition}{Definition}[section] %for article style
\newtheorem{theorem}[definition]{Theorem}
\newenvironment{theorem*}[1]{{\bf Theorem #1} \begin{itshape}}{\end{itshape}}
\newtheorem{lemma}[definition]{Lemma}

\newenvironment{corollary*}[1]{{\bf Corollary #1} \begin{itshape}}{\end{itshape}}

\newenvironment{proposition*}[1]{{\bf Proposition #1} \begin{itshape}}{\end{itshape}}

\theoremstyle{remark}
\newtheorem{remark}[definition]{Remark}
\usepackage{geometry}
\newcommand{\N}{\mathbb{N}}

\newcommand{\Z}{\mathbb{Z}}
\newcommand{\Q}{\mathbb{Q}}
\newcommand{\R}{\mathbb{R}}

\newcommand{\Zp}{\mathbb{Z}_{p}}

\newcommand{\cF}{\mathcal{F}}

\newcommand{\cC}{\mathcal{C}}

\newcommand{\cH}{\mathcal{H}}

\newcommand{\cQ}{\mathcal{Q}}
\newcommand{\cN}{\mathcal{N}}

\newcommand{\ba}{\mathbf{a}}

\newcommand{\bv}{\mathbf{v}}
\newcommand{\bp}{\mathbf{p}}

\newcommand{\bi}{\mathbf{i}}
\newcommand{\bt}{\boldsymbol{\tau}}

\newcommand{\bx}{\textbf{x}}
\newcommand{\bu}{\textbf{u}}
\newcommand{\dimh}{\dim_{\mathcal{H}}}
\usepackage{parskip}
\title{liminf approximation sets for abstract rationals}

\author[Mumtaz Hussain]{Mumtaz Hussain}
\address{Mumtaz Hussain,  Department of Mathematical and Physical Sciences,  La Trobe University, Bendigo 3552, Australia. }
\email{m.hussain@latrobe.edu.au}

\author{ Benjamin Ward}
\address{Ben Ward,  Department of Mathematical and Physical Sciences,  La Trobe University, Bendigo 3552, Australia. }
\email{Ben.Ward@latrobe.edu.au}
\date{\today}
\date{\today}

\begin{document}

\maketitle
\begin{abstract}
  The Jarn\'ik-Besicovitch theorem is a fundamental result in metric number theory which concerns the Hausdorff dimension for certain limsup sets. We discuss the analogous problem for liminf sets. Consider an infinite sequence of positive integers, $S=\{q_{n}\}_{n\in\N}$, exhibiting exponential growth. For a given $n$-tuple of functions denoted as $\Psi:=~(\psi_1, \ldots,\psi_n)$, each of the form $\psi_{i}(q)=q^{-\tau_{i}}$ for $(\tau_{1},\dots,\tau_{n})\in\R^{n}_{+}$, we calculate the Hausdorff dimension of the set of points that can be $\Psi$-approximated for all sufficiently large $q\in S$. We prove this result in the generalised setting of approximation by abstract rationals as recently introduced by Koivusalo, Fraser, and Ramirez (LMS, 2023). Some of the examples of this setting include the real weighted inhomogeneous approximation, $p$-adic weighted approximation, Diophantine approximation over complex numbers, and approximation on missing digit sets.
\end{abstract}
%\tableofcontents
\section{Introduction}
Dirichlet's Theorem in Diophantine approximation asserts that for any $n$-tuple of positive real numbers $\bt=(\tau_{1},\dots,\tau_{n})$ with $\tau_{1}+\dots+\tau_{n}=1$, every $\bx=(x_{1},\dots, x_{n})\in [0,1]^{n}$, and any $N\in\N$ there exists $1\leq q\leq N$ such that
\begin{equation*}
    \| q x_{i}\|<N^{-\tau_{i}} \quad (1\leq i \leq n)\, ,
\end{equation*}
where $\|\cdot\|$ denotes the minimum distance to the nearest integer. A corollary one can deduce from Dirichlet's theorem is that for every irrational $\bx\in [0,1]^{n}$ there exists an infinite sequence of integers $\{q_{j}(\bx)\}_{j\in\N}$ such that
\begin{equation*}
     \| q_{j}(\bx)x_{i}\|<q_{j}(\bx)^{-\tau_{i}} \quad (1\leq i\leq n)\, ,
\end{equation*}
for every $j \in \N$. In order to keep the statement true for all $\bx \in [0,1]^{n}$ one cannot improve on the size of the summation of the components of $\bt$. Indeed, if
\begin{equation*}
    \tau_{1}+\dots + \tau_{n} >1\, ,
\end{equation*}
then an easy application of the Borel-Cantelli lemma implies that Lebesgue-almost no points in $[0,1]^{n}$ can be $\bt$-approximated by rationals infinitely often. However, the following was proven by Rynne \cite{R98}. \\
\begin{theorem}[Rynne, 1998]
For $n$-tuple $\bt$ with $\tau_{1}+\dots + \tau_{n}>1$ then
\begin{equation*}
    \dimh \left\{ x\in[0,1]^{n}:\begin{array}{c}\|q x_{i}\|<q^{-\tau_{i}} \quad (1\leq i \leq n) \\ \text{ for infinitely many } q \in \N \end{array} \right\}=\min_{1\leq j \leq n} \left\{\frac{n+1+\sum\limits_{i:\tau_{j}>\tau_{i}}(\tau_{j}-\tau_{i})}{\tau_{j}+1}\right\}\, ,
\end{equation*}
where $\dimh$ denotes the Hausdorff dimension.
\end{theorem}
For the definition of Hausdorff measure and dimension see \S~\ref{prelims}. This result is the weighted analogue of the classical Jarn\'ik-Besicovitch theorem in Diophantine approximation \cite{B34, J29}. \par 
In this article, we consider the reverse of the setup given above. That is, given an infinite sequence of positive integers $\{q_{j}\}_{j\in\N}$ and an $n$-tuple $\bt=(\tau_{1},\dots,\tau_{n})\in\R^{n}_{+}$ with $\tau_{1}+\dots+\tau_{n}\geq~1$ for how many $\bx \in [0,1]^{n}$ does it hold that
\begin{equation*}
     \| q_{j}x_{i}\|<q_{j}^{-\tau_{i}} \quad (1\leq i \leq n)\, ,
\end{equation*}
for all sufficiently large $j\in\N$. This set can trivially be seen to be a null set in terms of Lebesgue measure (even when $\tau_{1}+\dots +\tau_{n}=1$) by considering the measure of each individual layer and taking the limit as $j \to \infty$. For the Hausdorff dimension, a corollary of our main result (stated in subsection \ref{mainresults}) gives us the following. 

\begin{theorem} \label{interesting?}
Let $\{q_{j}\}_{j\in\N}$ be an infinite sequence of positive integers and $\bt=(\tau_{1},\dots, \tau_{n})\in~\R^{n}_{+}$ an $n$-tuple of positive real numbers. Suppose that
\begin{equation} \label{difference_interesting}
    \lim_{j\to \infty} \frac{\log q_{j}}{\log q_{j-1}}=k>1
\end{equation}
exists and $\max\limits_{1\leq i \leq n}\tau_{i}<k-1<\infty$. Then
\begin{equation*}
   \dimh \left\{ \bx\in [0,1]^{n}: \hspace{-0.2cm}\begin{array}{l} \|q_{j}x_{i}\|<q_{j}^{-\tau_{i}} \quad (1\leq i \leq n) \\ \text{ \rm for all sufficiently large } j \in \N \end{array} \right\} = \min_{1\leq j \leq n}\left\{\frac{n-\frac{1}{k-1}\sum\limits_{i=1}^{n}\tau_{i} +\sum\limits_{i:\tau_{j}>\tau_{i}}(\tau_{j}-\tau_{i})}{\tau_{j}+1}\right\} \, .
   %\frac{n}{n+1}\left(n-\frac{1}{k-1}\right).
\end{equation*}
\end{theorem}
\begin{remark}
    As an idea of the sorts of sequences $S$ one can choose satisfying \eqref{difference_interesting} note that for any $(a,b)\in\N_{>1}^{2}$ the sequence $\{a^{b^{c}}\}_{c\in\N}$ has
    \begin{equation*}
        \lim_{j\to \infty} \frac{\log q_{j}}{\log q_{j-1}}=b\, .
    \end{equation*}
    Let
    \begin{equation*}
        \cF=\left\{ \{a^{b^{c}}\}_{c \in \N} : (a,b)\in\N_{>1}^{2} \right\}\, . 
    \end{equation*}
    Then Theorem~\ref{interesting?} and the countable stability of the Hausdorff dimension (see Section~\ref{prelims} for more details) gives us that for any $n$-tuple $\bt$ of finite positive real numbers 
    \begin{equation*}
        \dimh \left\{ \bx\in [0,1]^{n}: \, \exists\,  S \in \cF \text{ s.t. } \hspace{-0.2cm}\begin{array}{l} \|q_{j}x_{i}\|<q_{j}^{-\tau_{i}} \quad (1\leq i \leq n) \\ \text{ \rm for all sufficiently large } j \in \N \end{array} \right\} = \min_{1\leq j \leq n}\left\{\frac{n +\sum\limits_{i:\tau_{j}>\tau_{i}}(\tau_{j}-\tau_{i})}{\tau_{j}+1}\right\} \, .
    \end{equation*}
\end{remark}
 \begin{remark}\rm 
 Compared to the Theorem of Rynne, note that the set defined in Theorem~\ref{interesting?} is of strictly smaller dimension, even if the sequence $\{q_{j}\}_{j\in\N}$ is taken to be increasingly sparse. In fact, one can deduce from \cite{R98} that for $n$-tuple $\bt$ with $\tau_{1}+\dots +\tau_{n}>0$ and sequence $S$ satisfying \eqref{difference_interesting} then
 \begin{equation*}
     \dimh \left\{ \bx \in [0,1]^{n}: \begin{array}{c}
     \|q_{j}x_{i}\|<q_{j}^{-\tau_{i}} \quad (1\leq i \leq n)\\
     \text{ for infinitely many } j\in \N \end{array} \right\} = \min_{1\leq j \leq n}\left\{\frac{n +\sum\limits_{i:\tau_{j}>\tau_{i}}(\tau_{j}-\tau_{i})}{\tau_{j}+1}\right\} \, .
 \end{equation*}
 This follows from \cite[Theorem 1]{R98} and the observation that the value $v(Q)$ ($v(S)$ in our case) appearing in \cite[Theorem 1]{R98} is equal to zero for the sequences we are considering.
 \end{remark}
 \begin{remark}\rm Theorem \ref{interesting?} is an extension of the recent work of the first-named author and Shi who proved the result in the non-weighted setting \cite{HussainShi23}. Theorem~\ref{interesting?} is a corollary of our main result which we state in the next section. In particular, we do not require the limit \eqref{difference_interesting} to exist, we only need to calculate the $\liminf$ of the sequence. See Theorem~\ref{main} and Theorem~\ref{real_corollary} for more details.
 \end{remark}
 \par 
 % Note that as the limit \eqref{difference_interesting} tends to infinity the above set does not have full dimension. In the one-dimensional case we obtain the dimension to be $\frac{1}{2}$. {\color{red}( I'm not too sure on this) This is not a surprise! Let $a_n$ be the $n$th partial quotient and $q_n$ be the denominator of the $n$th convergent. It is well-known that, for any $\tau>0$, the set of uniformly approximable numbers
 % $$\left\{x\in[0, 1): \|qx\|<q^{-(\tau+2)} \ \text{for all } q\in\N\right\}$$
 % is equivalent to the set
 %  $$\left\{x\in[0, 1): a_n(x)\geq q_n^\tau \ \text{for all } n\in\N\right\}.$$
 %  Within the context of our problem, the $n$th convergents will also obey \eqref{difference_interesting}, therefore, when $k\to \infty$ we note that $a_n(x)\to \infty$. Hence by Good's theorem (1941), the Hausdorff dimension of the corresponding set is $1/2$.}
 In order to state our main theorem we recall and expand upon the definition of \textit{abstract rationals} as introduced in \cite{FKR23}. We then recall the notation of \cite{HussainShi23} on the exponential shrinking problem in the non-weighted case and redefine these sets for abstract rationals. Our main results are then given, which are followed by a series of applications. These include the real weighted inhomogeneous approximation, $p$-adic weighted approximation, complex Diophantine approximation, and approximation on missing digit sets. In \S 3-5 we prove our main results.

\subsection{Abstract rationals}
Fix $n\in\N$ and for each $1\leq i \leq n$, let $(F_{i},d_{i},\mu_{i})$ be a non-empty totally bounded metric space equipped with a $\delta_{i}$-Ahlfors regular measure $\mu_{i}$ with support $F_{i}$. That is, for any $x_{i}\in F_{i}$ and $0<r<r_{0}$ for some bounded $r_{0}$ there exists constants $0<c_{1,i}\leq c_{2,i}< \infty$ such that
\begin{equation*}
    c_{1,i} r^{\delta_{i}} \leq \mu_{i}(B_{i}(x_{i},r)) \leq c_{2,i} r^{\delta_{i}},
\end{equation*} 
where for any $x_{i}\in F_{i}$ and $r>0$ we write
\begin{equation*}
    B_{i}(x_{i},r)=\{y\in F_{i}: d_{i}(x_{i},y)<r\}\, .
\end{equation*}
Let
\begin{equation*}
    F=\prod_{i=1}^{n}F_{i}, \quad d(\cdot,\cdot)=\max_{1\leq i\leq n} d_{i}(\cdot,\cdot), \quad \mu=\prod_{i=1}^{n} \mu_{i}
\end{equation*}
so that $(F,d,\mu)$ is the product metric space. For any two subsets $A,B\subset F_{i}$ by $d_{i}(A,B)$ we mean
\begin{equation*}
    d_{i}(A,B)=\inf\{d_{i}(a,b): a\in A, b\in B\},
\end{equation*}
and for any $\bx=(x_{1},\dots, x_{n}) \in F$ and $r>0$ we write
\begin{equation*}
    B((x_{1},\dots, x_{n}),r)=\prod_{i=1}^{n}B_{i}(x_{i},r).
\end{equation*}
Let $\cN$ be an infinite countable set, and let $\beta:\cN \to \R_{+}$ be function which associates a weight $\alpha \mapsto \beta(\alpha)=\beta_{\alpha}$. In line with \cite{FKR23} For each $1\leq i\leq n$ and each $q\in\cN$ define the \textit{$\beta$-abstract rationals of level $q$ in $F_{i}$} by fixing a \textit{maximal $\beta_{q}^{-1}$-separated} set of points $P_{i}(q) \subset F_{i}$, where
\begin{itemize}
    \item \textit{$\beta_{q}^{-1}$-separated} means that for all $p_{1},p_{2} \in P_{i}(q)$ we have that $d_{i}(p_{1},p_{2})\geq \beta_{q}^{-1}$, and
    \item \textit{maximal} means that for all $x \in F_{i}$ there exists $p \in P_{i}(q)$ such that $d_{i}(p,x)<\beta_{q}^{-1}$.
\end{itemize}
That is, all the points are reasonably separated from each other and we cannot fit any more level $q$ abstract rationals without being too close to an already present abstract rational.
Let $P(q)=\prod_{i=1}^{n}P_{i}(q)$ denote the product of abstract rationals of level $q$ in $F$, and let 
\begin{equation*}
    \cQ=\bigcup_{q\in\cN}P(q)
\end{equation*}
be the set of \textit{$\beta$-abstract rationals in $F$}.\par 
As an example consider 
\begin{align*}
    F_{i}=&\,[0,1]^{d}\, , \quad  d_{i}=|\cdot| \text{ the maximum norm}, \\
    \cN=&\,\N\, , \quad \quad  \beta(q)=q\, , \\
    P_{i}(q)&=\left\{\left(\frac{p_{1}}{q},\dots,\frac{p_{d}}{q}\right): 0\leq p_{1}, \dots, p_{d} \leq q\right\}=\frac{1}{q}\Z^{d} \cap [0,1]^{d} \, , \quad  (1\leq i \leq n)\, .
\end{align*}
%$F_{i}=[0,1]$, $d_{i}=|\cdot|$ the absolute norm, $\cN=\N$, $\beta(q)=q$, and $P_{i}(q)=\left\{\frac{p}{q}: 0\leq p \leq q\right\}$, for each $1\leq i \leq n$. 
Then $\cQ=\Q^{dn}\cap [0,1]^{dn}$. Generally one can take each $F_{i}$ to be any bounded convex body in $\R^{d}$, and the lattice $\frac{1}{q}\Z^{d}$ can be shifted by any vector.

\subsection{Liminf approximation sets}
In \cite{HussainShi23}, Hussain and Shi introduced an exponential shrinking problem stated as follows. Given a sequence $S=\{q_{j}\}_{j\in\N}$ of positive integers, fixed $\theta=(\theta_{1},\dots, \theta_{n}) \in [0,1]^{n}$, and $\tau \geq 1$, consider the set
\begin{equation*}
    \Lambda_{\Q^{n}}^{S}(\tau):=\left\{ \bx \in [0,1]^{n}: \max_{1\leq i\leq n}||q_{j}x_{i}-\theta_{i}||<q_{j}^{-\tau} \quad \text{ for all } j\geq 1 \right\}.
\end{equation*}
This set was introduced to answer a question related to a problem posed by Schleischitz in \cite{SchleischitzSelecta}. Under certain conditions on $\tau$ and $S$, they provide the exact Hausdorff dimension of the set $\Lambda_{\Q^{n}}^{S}(\tau)$. In this article, we generalise the above setting by considering weighted approximation (the approximation function can vary between coordinate axes) by abstract rationals (which includes approximation by rationals as a special case). \par 
Fix a sequence $S=\{q_{j}\}_{j\in\N}$ with each $q_{j}\in\cN$ and weight vector $\bt=(\tau_{1},\dots, \tau_{n})\in\R^{n}_{+}$. Define the set
\begin{equation*}
    \Lambda_{\cQ}^{S}(\bt):=\left\{ \bx\in F : \begin{array}{c} d_{i}(x_{i},P_{i}(q_{j}))<\beta(q_{j})^{-\tau_{i}} \quad (1\leq i \leq n) \\ \text{  for all } \, j \in \N \end{array} \right\},
\end{equation*}
and the set
\begin{equation*}
    \widehat{\Lambda}_{\cQ}^{S}(\bt):=\left\{ \bx\in F : \begin{array}{c}d_{i}(x_{i},P_{i}(q_{j}))<\beta(q_{j})^{-\tau_{i}} \quad (1\leq i \leq n)\\ \text{  for all sufficiently large } \, j \in \N \end{array} \right\}.
\end{equation*}

 The latter set is a slight relaxation of the former set since we only require the points to be eventually always close to a sequence of abstract rationals. We may write the latter set in terms of the former by defining for any $t\in\N$
\begin{equation*}
    \sigma^{t} S=\{q_{i+t}\}_{i\in \N},
\end{equation*}
that is, $\sigma$ is the left shift on sequence $S$. Then
\begin{equation*}
    \widehat{\Lambda}_{\cQ}^{S}(\bt) = \bigcup_{t\in\N} \Lambda_{\cQ}^{\sigma^{t} S}(\bt).
\end{equation*}

\subsection{Main results}\label{mainresults}
We prove the following results on the Hausdorff dimension of the sets introduced above. 

\begin{theorem} \label{main}
    Let $(F,d,\mu)$ be a product space of non-empty totally bounded metric spaces each equipped with Ahlfors regular measures. Let $\cN$ be an infinite countable set $\cN$, $\beta:\cN\to \R_{+}$, and $\cQ$ be a set of $\beta$-abstract rationals in $F$. Fix an infinite sequence $S$ contained in $\cN$ over which $\beta$ is unbounded and strictly increasing. Suppose that
    \begin{equation*}
        \inf_{j\in \N} \frac{\log \beta(q_{j})}{\log \beta(q_{j-1})}=h_{S}>1, \, \, \text{   and   } \,\, \liminf_{j\to \infty} \frac{\sum\limits_{i=1}^{j-1}\log \beta(q_{i})}{\log \beta(q_{j})}=\alpha_{S}.
    \end{equation*}
    For any $\bt$ such that $h_{S}>\tau_{i}>1$ for each $1\leq i \leq n$, we have that
    \begin{equation*}
        \dimh \Lambda_{\cQ}^{S}(\bt)=\min_{1\leq k \leq n}\left\{ \frac{1}{\tau_{k}}\left(\sum_{i=1}^{n}\delta_{i} -\alpha_{S}\sum_{i=1}^{n}(\tau_{i}-1)\delta_{i} + \sum_{j:\tau_{k}\geq \tau_{j}}(\tau_{k}-\tau_{j})\delta_{j} \right)\right\}.
    \end{equation*}
\end{theorem}
\begin{remark}\rm
The infimum condition of Theorem~\ref{main} cannot be replaced by a $\liminf$ condition. To see this consider for example $F=[0,1]$, $\beta(q)=q$, $$\cQ=\left\{\frac{p+\theta}{q}: (p,q)\in\Z\times \N \text{ and } \frac{p+\theta}{q} \in [0,1]\right\},$$ and the sequence $$S=\{2,3,3^{h},3^{h^{2}}, \dots \}.$$ For large $h$ and suitable choice of $\theta\in [0,1]$ it can be shown that the sets of points
\begin{align*}
    &\{x\in[0,1]:\|2x+\theta\|<2^{-h+1}\}  \text{  and  } \\
    &\{x\in[0,1]:\|3x+\theta\|<3^{-h+1}\} 
\end{align*}
are disjoint and so $\Lambda_{\cQ}^{S}(\bt)=\emptyset$.
\end{remark}
\begin{remark}\rm
In order to state our result for the general case of abstract rationals, it is necessary to suppose each $\tau_{i}<h_{S}$. Without this assumption additional considerations are necessary. See \cite[Remark 1.2]{HussainShi23} for a discussion on complications in the case of real simultaneous approximation (non-weighted).
\end{remark}

 The infimum condition can be replaced by a $\liminf$ condition if we consider the set $\widehat{\Lambda}_{\cQ}^{S}(\bt)$.

\begin{theorem} \label{corollary}
 Let $(F,d,\mu)$, $\cN$, $\beta$, and $\cQ$ be constructed as above. Fix an infinite sequence $S$ contained in $\cN$ over which $\beta$ is unbounded and strictly increasing and suppose that
    \begin{equation*}
        \liminf_{j\to \infty} \frac{\log \beta(q_{j})}{\log \beta(q_{j-1})}=h>1, \, \, \text{   and   } \,\, \liminf_{j\to \infty} \frac{\sum\limits_{i=1}^{j-1}\log \beta(q_{i})}{\log \beta(q_{j})}=\alpha_{S}.
    \end{equation*}
    For any $\bt$ such that $h>\tau_{i}>1$ for each $1\leq i \leq n$, we have that
    \begin{equation*}
        \dimh \widehat{\Lambda}_{\cQ}^{S}(\bt)=\min_{1\leq k \leq n}\left\{ \frac{1}{\tau_{k}}\left(\sum_{i=1}^{n}\delta_{i} -\alpha_{S}\sum_{i=1}^{n}(\tau_{i}-1)\delta_{i} + \sum_{j:\tau_{k}\geq \tau_{j}}(\tau_{k}-\tau_{j})\delta_{j} \right)\right\}.
    \end{equation*}
\end{theorem}

It should be noted that Theorem~\ref{corollary} essentially follows from Theorem~\ref{main} and the countable stability of the Hausdorff dimension. For completeness, we provide the proof at the end of \S \ref{Section:CorollaryProof}.

\medskip

\noindent{\bf Acknowledgments:} The research of both authors is supported by the Australian Research Council discovery project 200100994.

%%%%%%%%%%%%%%%%%%%%%%%%%%%%%%%%%%%%%%
%
%   APPLICATIONS
%
%%%%%%%%%%%%%%%%%%%%%%%%%%%%%%%%%%%%%%

\section{Applications}
 We begin with the classical setting of real approximation by rational numbers, which we generalise to the weighted inhomogeneous setting. We then give similar statements in the case of $p$-adic approximation. In later applications, we give the statement in the simplified one-dimensional homogeneous setting. It should be clear from the application in the real weighted inhomogeneous setting that the one-dimensional case readily generalises to the higher dimensional weighted setting. Indeed, the only calculations required to apply Theorems~\ref{main}--\ref{corollary} is to show our setup aligns with some approximation by abstract rationals. Once this is done in one dimension it is clear that the product space also satisfies the criteria of abstract rationals. \par
 %{\color{red} (does this seem reasonable? After reading the real weighted setting it should be obvious how the weighted setting works. If not can you please add the weighted analogues in each case)}{\color{blue}real and p-adic are weighted but all other are only in the one dimension settings. This requires some work}. \par 
 The notion of abstract rationals admits a broad range of applications, though in some instances it is not immediately clear whether a set satisfies the properties of being abstract rationals. The main work in each of these applications is constructing the sets of abstract rationals. The list of applications presented below is far from exhaustive. For example, in increasing levels of difficulty, one could consider formal power series approximation, approximation by irrational rotations, and approximation of real manifolds by rational numbers. The latter two cases seem particularly challenging. \par 

%%%%%%%%%%%%%%%%%%%%%%%%%%%%%%%%%%%%%%
%
%   REAL
%
%%%%%%%%%%%%%%%%%%%%%%%%%%%%%%%%%%%%%%

\subsection{Real approximation} 
The classical study of approximation of real numbers by rationals is extensive, see \cite{BRV16} for a survey of the foundational results and \cite{HY14,KW23,KTV06,R98, WW19} for the more recent weighted analogies of such results. In this section let 
\begin{align*}
    F_{i}=[0,1]\, , \quad d_{i}=&|\cdot|\, , \quad \mu_{i}=\lambda\, , \\
    \text{ so } \, (F,d,\mu)=&([0,1]^{n},|\cdot|,\lambda_{n})\, , 
\end{align*}
for $|\cdot|$ the usual max norm on real space, $\lambda$ the Lebesgue measure, and $\lambda_{n}$ the $n$-dimensional Lebesgue measure. Let
\begin{align*}
    \cN=&\N\, , \quad  \quad  \beta(q)=q\, , \quad \quad   \theta:\cN \to [0,1]^{n}\, , \text{ and } \\
    P_{i}(q)=&\left\{ \frac{p+\theta_{i}(q)}{q}: p \in\Z \, \, \,   \text{ and }  \,  \,  \frac{p+\theta_{i}(q)}{q}\in [0,1] \right\} \, , \, \, \,  ( 1 \leq i \leq n )
\end{align*}
be the  $q^{-1}$-abstract rationals of level $q$ in $[0,1]$. \par 
Observe that each $P_{i}(q)$ can be seen as a subset of the shifted lattice $\tfrac{1}{q}\Z + \theta_{i}(q)$, thus it is clear each point is $q^{-1}$ separated, and furthermore the set is maximal. Note we only have to show that each $P_{i}(q)$ is a well defined set of $q^{-1}$-abstract rationals of level $q$ in $[0,1]$. The higher dimensional product space result follows immediately. Hence
\begin{equation*}
    \cQ= \bigcup_{q\in\N} \prod_{i=1}^{n}\left\{ \frac{p+\theta_{i}(q)}{q}: p \in\Z \, \, \,   \text{ and }  \,  \,  \frac{p+\theta_{i}(q)}{q}\in [0,1] \right\}
\end{equation*}
is a well-defined set of abstract rationals.  \par 
For $\theta(q)=0$ for all $q\in \N$ this is the standard homogeneous setting, and for $\theta(q)=(\theta_{1},\dots, \theta_{n})$ fixed this is the standard inhomogeneous setting. \par 
Let $S=\{q_{i}\}_{i\in\N}$ be an increasing sequence of positive integers and define the sets
\begin{align*}
    W_{n}^{S}(\bt)=\left\{ \bx \in [0,1]^{n} : \begin{array}{c}\left\| q_{j}x_{i}-\theta(q_{j})_{i}\right\|<q_{j}^{-\tau_{i}}\, \quad (1\leq i \leq n) \\
    \quad \text{ for all } j \in \N \end{array} \right\}\, , \\
     \widehat{W}_{n}^{S}(\bt)=\left\{ \bx \in [0,1]^{n} : \begin{array}{c}\left\| q_{j}x_{i}-\theta(q_{j})_{i}\right\|<q_{j}^{-\tau_{i}}\, \quad (1\leq i \leq n) \\
    \quad \text{ for all sufficiently large } j \in \N \end{array} \right\}\, .
\end{align*}
Note that dividing through by $q_{j}$ in the inequalities in $W_{n}^{S}(\bt)$ gives us $W_{n}^{S}(\bt)=\Lambda_{\cQ}^{S}(\bt+1)$, and similarly $\widehat{W}_{n}^{S}(\bt)=\widehat{\Lambda}_{\cQ}^{S}(\bt+1)$. Notice by choice $\beta(q)=q$ for any strictly increasing sequence of positive integers $S$ we immediately have that $\beta$ is unbounded and strictly increasing. Applying Theorem~\ref{main} to this setting we immediately have the following.
\begin{theorem} \label{real_corollary}
    Let $S$ be an increasing sequence of integer with
    \begin{equation*}
        \inf_{j\to \infty} \frac{\log q_{j}}{\log q_{j-1}}=h_{S}>1, \, \, \text{   and   } \,\, \liminf_{j\to \infty} \frac{\sum\limits_{i=1}^{j-1}\log q_{i}}{\log q_{j}}=\alpha_{S}.
    \end{equation*}
    For any $\bt$ such that $h_{S}-1>\tau_{i}>0$ for each $1\leq i \leq n$, we have that
    \begin{equation*}
        \dimh W_{n}^{S}(\bt)=\min_{1\leq k \leq n} \left\{\frac{1}{\tau_{k}+1}\left(n-\alpha_{S}\sum_{i=1}^{n}\tau_{i} + \sum_{i:\tau_{k}\geq \tau_{i}} (\tau_{k}-\tau_{i}) \right)\right\}.
    \end{equation*}
\end{theorem}
\begin{remark} \rm 
This result is a generalisation of \cite[Theorem 1.1]{HussainShi23} to the weighted setting, where it was proven that for $\tau=\tau_{1}=\dots=\tau_{n}$ with $h_{S}-1>\tau>0$ that
\begin{equation*}
    \dimh W^{S}_{n}(\tau)=\frac{n}{\tau+1}(1-\alpha_{S}\tau),
\end{equation*}
which agrees with the theorem above.
Unlike \cite{HussainShi23}, in our setting the inhomogeneity $\theta$ is allowed to vary over the sequence $S$.
\end{remark}
By applying Theorem~\ref{corollary} instead of Theorem~\ref{main} we have the following result.
\begin{theorem} \label{real_corollary2}
    Let $S$ be an increasing sequence of integer with
    \begin{equation*}
        \liminf_{j\to \infty} \frac{\log q_{j}}{\log q_{j-1}}=h_{S}>1, \, \, \text{   and   } \,\, \liminf_{j\to \infty} \frac{\sum\limits_{i=1}^{j-1}\log q_{i}}{\log q_{j}}=\alpha_{S}.
    \end{equation*}
    For any $\bt$ such that $h_{S}-1>\tau_{i}>0$ for each $1\leq i \leq n$, we have that
    \begin{equation*}
        \dimh W_{n}^{S}(\bt)=\min_{1\leq k \leq n} \left\{\frac{1}{\tau_{k}+1}\left(n-\alpha_{S}\sum_{i=1}^{n}\tau_{i} + \sum_{i:\tau_{k}\geq \tau_{i}} (\tau_{k}-\tau_{i}) \right)\right\}.
    \end{equation*}
\end{theorem}

\vspace{1ex}

In later applications we will only give a results aligning to setups of the form $\Lambda_{\cQ}^{S}(\bt)$. It should be clear that the statements relating to $\widehat{\Lambda}_{\cQ}^{S}(\bt)$ follow immediately.

%%%%%%%%%%%%%%%%%%%%%%%%%%%%%%%%%%%%%%
%
%   P-ADIC
%
%%%%%%%%%%%%%%%%%%%%%%%%%%%%%%%%%%%%%%

\subsection{$p$-adic weighted approximation}
Fix a prime number $p$ and let
\begin{align*}
    F=\Zp\, , \quad d=&|\cdot|_{p}\, , \quad \mu=\mu_{p}\, , \\
\end{align*}
for $\Zp$ the ring of $p$-adic integers, $|\cdot|_{p}$ the $p$-adic norm, and $\mu_{p}$, the $p$-adic Haar measure normalised at $\mu_{p}(\Zp)=1$. For metric properties of the classical sets of Diophantine approximation in $p$-adic space see \cite{A95, J45,L55}, and for the more recent weighted setting see \cite{BLW21b,GHSW23, KTV06}. \par 
Due to the ultrametric properties of $p$-adic space, it is slightly more complicated to construct layers of abstract rationals. We opt for the following setup. Let 
\begin{equation*}
    \cN=\{p^{k}: k\in\N \}\, , \quad \text{ and } \quad P(q)=\left\{\frac{a}{q-1}: 1\leq a \leq q \right\}\, .
\end{equation*}
Note that for $\frac{a}{q-1},\frac{a'}{q-1}\in P(q)$ with $a\neq a'$ we have that
\begin{equation*}
    \left| \frac{a}{q-1}-\frac{a'}{q-1}\right|_{p}=|q-1|_{p}^{-1}|a-a'|_{p}= |a-a'|_{p} \geq p^{-k}=|q|^{-1},
\end{equation*}
and so $P(q)$ is $q^{-1}$-separated for each $q\in\cN$. Thus define $\beta(q)=|q|$. \par 
To show each $\beta$-abstract rationals of level $q$ is maximal observe that there are $p^{k}$ points contained in $P(p^{k})$ and each of these lies in $\Zp$ by the coprimality of $(p^{k}-1,p)$. The $p$-adic max norm balls 
\begin{equation*}
    \bigcup_{\frac{a}{p^{k}-1}\in P_{i}(p^{k})}B\left(\frac{a}{p^{k}-1}, p^{-k}\right)
\end{equation*}
are disjoint and furthermore are a cover of $\Zp$. To see this write each $x\in\Zp$ as their $p$-adic expansion
\begin{equation} \label{p-adic cover}
    x=\sum_{i=1}^{\infty}x_{i}p^{-i} \quad x_{i}\in \{0,1,\dots , p-1\}.
\end{equation}
Since each $\frac{a}{p^{k}-1} \in P(p^{k})$ is $p^{-k}$-separated, each of their corresponding $p$-adic expansions must differ over the first $k$ coefficients. There are $p^{k}$ different $p$-adic expansions over the first $k$ coefficients so the $p^{k}$ abstract rationals of level $p^{k}$ cover each possible expansion. Thus any $x\in \Zp$ belongs to some ball in the union \eqref{p-adic cover}, and so every $x\in\Zp$ is $p^{-k}$ close to an abstract rational of level $p^{k}$. \par
For each $q\in\cN$ let $\cQ_{p}=\bigcup_{q\in\cN}P(q) \subset \Zp$. For $\tau\in\R_{+}$ and sequence $S=\{p^{k_{j}}\}_{j\in\N}$ with $k_{j}\in\N$ an increasing sequence define
\begin{equation*}
    \mathcal{W}_{\Zp}^{S}(\tau)=\left\{ x \in \Zp: \left|x-\frac{a}{p^{k_{j}}-1}\right|_{p}<p^{-k_{j}\tau} \, \  \text{ for some } \frac{a}{p^{k_{j}}-1}\in P(p^{k_{j}}) \text{ and for all } j\in \N \right\}
\end{equation*}
Note that $\mathcal{W}_{\Zp}^{S}(\tau)=\Lambda_{\cQ_{p}}^{S}(\tau)$, and observe that since the sequence $S$ is strictly increasing $\beta(q)=|q|$ is unbounded and strictly increasing. Hence applying Theorem~\ref{main} we have the following.
\begin{theorem}
    Let $S=\{p^{k_{j}}\}_{j\in\N}$ be a sequence of positive integers with $k_{j}\in\N$ increasing and let 
    \begin{equation*}
        \inf_{j\in \N} \frac{k_{j}}{k_{j-1}}=h_{S}>1, \, \, \text{   and   } \,\, \liminf_{j\to \infty} \frac{\sum\limits_{i=1}^{j-1}k_{i}}{ k_{j}}=\alpha_{S}.
    \end{equation*}
     For any $\tau$ such that $h_{S}>\tau>1$ we have that
        \begin{equation*}
            \dimh \mathcal{W}_{\Zp}^{S}(\bt) = \frac{1}{\tau}\left(1-\alpha_{S}(\tau-1) \right)\, .
        \end{equation*}
\end{theorem}

For completeness in this application, we also provide the weighted result. Redefine
\begin{align*}
    F&=\Zp^{n}\, ,\quad  \quad d=\max_{1\leq i \leq n} |\cdot |_{p}\, ,\quad \quad \mu=\mu_{p,n}=\prod_{i=1}^{n}\mu_{p}\, ,\\
    \cN&=\{p^{k}: k\in\N\} \, , \quad P(q)=\prod_{i=1}^{n}P_{i}(q)=\prod_{i=1}^{n}\left\{\frac{a_{i}}{q-1}:1\leq a\leq q\right\}\, ,\\
    \cQ_{p,n}&=\bigcup_{q\in\cN}P(q)\subset \Zp^{n}\, .
\end{align*}
By our above calculations in the one dimensional case each $P_{i}(q)$ are $\beta$-abstract rationals of level $q$ in $\Zp$ so $P(q)=\prod_{i=1}^{n}P_{i}(q)$ are $\beta$-abstract rationals of level $q$ in $\Zp^{n}$. For $n$-tuple $\bt\in\R^{n}_{+}$ and sequence $S=\{p^{k_{j}}\}_{j\in\N}$ with $k_{j}$ an increasing sequence let
\begin{equation*}
    \mathcal{W}_{\Zp^{n}}^{S}(\bt)=\left\{ \bx \in \Zp^{n}:\begin{array}{c} \left|x-\frac{a_{i}}{p^{k_{j}}-1}\right|_{p}<p^{-k_{j}\tau_{i}} \, \quad (1\leq i \leq n) \\  \text{ for some } \frac{\ba}{p^{k_{j}}-1}\in P(p^{k_{j}}) \text{ and for all } j\in \N \end{array} \right\}\,.
\end{equation*}
\begin{theorem}
    Let $S=\{p^{k_{j}}\}_{j\in\N}$ be a sequence of positive integers with $k_{j}\in\N$ increasing and let 
    \begin{equation*}
        \inf_{j\in \N} \frac{k_{j}}{k_{j-1}}=h_{S}>1, \, \, \text{   and   } \,\, \liminf_{j\to \infty} \frac{\sum\limits_{i=1}^{j-1}k_{i}}{ k_{j}}=\alpha_{S}.
    \end{equation*}
    For any $\bt$ such that $h_{S}>\tau_{i}+1>0$ for each $1\leq i \leq n$ we have that
        \begin{equation*}
            \dimh \mathcal{W}_{n}^{S}(\bt) = \min_{1\leq k \leq n} \left\{\frac{1}{\tau_{k}}\left(n-\alpha_{S}\sum_{i=1}^{n}(\tau_{i}-1) + \sum_{i:\tau_{k}\geq \tau_{i}} (\tau_{k}-\tau_{i}) \right)\right\}.
        \end{equation*}
\end{theorem} 
\vspace{2ex}

\subsection{Complex Diophantine approximation} 
The classical setting of Diophantine approximation of complex numbers by Gaussian integers has been studied by a range of authors, see for example \cite{BGH,HeXiong2021,HussainNZJM} and \cite[Sections 4-6]{DodsonKristensen}. In the complex case, we consider the following setup. Let 
\begin{align*}
    F=\mathfrak{F}=[-\tfrac{1}{2},\tfrac{1}{2}]\times[-\tfrac{1}{2},\tfrac{1}{2}]i\, &, \quad d=\|\cdot\|_{2}\, , \quad \mu=\lambda_{2}\, ,
\end{align*}
for $\|\cdot\|_{2}$ the Euclidean norm. Note that $\delta=2$ in this setting. Let    
\begin{align*}
    \cN=\{a+bi: a,b\in\Z\}\, , \quad  &\beta(a+bi)=\|a+bi\|_{2}=a^{2}+b^{2}\, .
\end{align*}

%$\mathfrak{F}=[-\frac{1}{2},\frac{1}{2}]\times[-\frac{1}{2},\frac{1}{2}]i\subset \mathbb{C}$, $\|\cdot\|_{2}$ the Euclidean norm, and let $\mu$ be associated to the two dimensional Lebesgue measure. Let $\cN=\{a+bi:a,b\in\Z\}$ be the set of Gaussian integers and $\beta(a+bi)=\|a+bi\|_{2}=a^{2}+b^{2}$.
Then, for any $q\in\cN$, the set of Gaussian integers, let 
$$P(q)=\left\{\frac{p}{q}: p \in \cN \text{ and  } \frac{p}{q} \in \mathfrak{F} \right\}\, . $$ This set can be associated to a lattice on $\R^{2}$ with Euclidean distance $\|\frac{p}{q}-\frac{p'}{q}\|_{2}\geq \|q\|_{2}^{-1}$, see \cite[Section 4.5]{DodsonKristensen}. Hence $P(q)$ is $\|q\|_{2}^{-1}$-separated and maximal, and so $\cQ_{\mathbb{C}}=\bigcup_{q\in\cN}P(q)$ is a well defined set of $\beta$-abstract rationals on $\mathcal{F}$. Let $S=\{q_{j}\}_{j\in \N}$ be a sequence of Gaussian integers with strictly increasing norm, and let $\tau \in \R_{+}$. Define the set
\begin{equation*}
    \mathfrak{W}_{\mathbb{C}}^{S}(\tau)=\left\{ z \in \mathfrak{F} : \left\|z-\frac{p}{q_{j}}\right\|_{2}< \|q_{j}\|_{2}^{-\tau-1} \text{ for some } \frac{p}{q_{j}}\in P(q_{j}) \text{ and for all } j\in\N \right\}.
\end{equation*}
Note by the condition that $S$ is a sequence of Gaussian integers with strictly increasing norm $\beta(q)=\|q\|_{2}$ is strictly increasing and unbounded. Furthermore note that $\mathfrak{W}_{\mathbb{C}}^{S}(\tau)=\Lambda_{\cQ_\mathbb{C}}^{S}(\tau)$. Applying Theorem~\ref{main} we have the following.
\begin{theorem}
    Let $S=\{q_{j}\}_{j\in\N}$ be a sequence of Gaussian integers with strictly increasing norm. Let $h_{S}$ and $\alpha_{S}$ be defined and satisfy the conditions as in Theorem~\ref{main}. For any $\tau\in\R_{+}$ such that $h_{S}>\tau>1$ we have that
    \begin{equation*}
        \dimh \mathfrak{W}_{\mathbb{C}}^{S}(\tau)=\frac{2}{\tau+1}\left(1-\alpha_{S}\tau \right).
    \end{equation*}
\end{theorem}

%%%%%%%%%%%%%%%%%%%%%%%%%%%%%%%%%%%%%%
%
%   MISSING DIGIT SETS
%
%%%%%%%%%%%%%%%%%%%%%%%%%%%%%%%%%%%%%%

\subsection{Approximation on missing digit sets}
Diophantine approximation on fractals has been an area of intense study, particularly since Mahler's 1984 paper \cite{M84} in which he asked how well points inside the middle-third Cantor set could be approximated by rational points either i) inside, or ii) outside of the middle-third Cantor set. This question has since been generalised significantly. For classical approximation results in this setting see \cite{AB21, SBaker,  FS14, KLW05,KW05, LSV07, TanWangWu} and \cite{KW23, WW19} for the higher dimensional weighted setting.   Let $b\in\N_{\geq 3}$ be fixed and let $J\subset \{0,1,\dots , b-1\}$ denote a proper subset with $\# J\geq 2$. For each $j\in J$ define the maps $f_{j}:[0,1] \to [0,1]$ by
\begin{equation*}
    f_{j}(x)=\frac{x+j}{b},
\end{equation*}
and let $\Phi=\{f_{j}:j\in J\}$. Consider the self-similar iterated function system $\Phi$ and let $\cC(b,J)$ be the attractor of $\Phi$. That is, $\cC(b,J)$ is the unique non-empty compact subset of $[0,1]$ such that
\begin{equation*}
    \cC(b,J)=\bigcup_{j\in J} f_{j}(\cC(b,J) ).
\end{equation*}
Call $\cC(b,J)$ the $(b,J)$-missing digit set. As an example note that $\cC(3,\{0,2\})$ is the classical middle-third Cantor set. These sets can be equipped with a self similar measure $\mu_{\cC}$ defined as 
\begin{equation*}
   \mu_{\cC}(B)=\mathcal H^{\gamma(b, J)}(B\cap \cC(b, J))\, ,
\end{equation*}
which was shown by Mauldin and Urbanski \cite{MU96} to be $\gamma(b,J)$-Ahlfors regular where
\begin{equation*}
    \gamma(b,J)=\dimh \cC(b,J)=\frac{\log\# J}{\log b}.
\end{equation*}
Concisely, let
\begin{align*}
    F=&\cC(b,J)\, , \quad d=|\cdot|\, , \quad \mu=\mu_{\cC}\, , \\
    \cN&=\N\, , \quad  \text{ and } \, \, \beta(k)=b^{k}\, .
\end{align*}
We now construct the abstract rationals in this setting. Henceforth, fix any point $z\in\cC(b,J)$ and for any $k\in \cN$ let
\begin{equation*}
    P(k)=\{ f_{\bi}(z) : \bi \in J^{k} \},
\end{equation*}
where $f_{\bi}=f_{i_{1}} \circ \dots \circ f_{i_{n}}$ for the finite word $\bi=(i_{1},\dots, i_{n})\in J^{n}$. Note that for each $f_{\bv}(z),f_{\bu}(z) \in P(k)$ with $\bv \neq \bu$ we have that
\begin{align*}
    |f_{\bv}(z)-f_{\bu}(z)|& =\left|\left(\sum_{i=1}^{k}v_{i}b^{-i} +\sum_{i=k+1}^{\infty}z_{i-k}b^{-i}\right) -\left(\sum_{i=1}^{k}u_{i}b^{-i} +\sum_{i=k+1}^{\infty}z_{i-k}b^{-i}\right) \right| \\
    &= \left|\sum_{i=1}^{k}(v_{i}-u_{i})b^{-i} \right|. 
\end{align*}
Since $\bu \neq \bv$ they must differ in at least one digit. Suppose that $(u_{1},\dots,u_{t})=(v_{1}, \dots , v_{t})$ but $u_{t+1}\neq v_{t+1}$ for some $0\leq t\leq k-1$. Then
\begin{align*}
    \left|\sum_{i=1}^{k-1}(v_{i}-u_{i})b^{-i} \right|&=\left|\sum_{i=t+1}^{k}(v_{i}-u_{i})b^{-i} \right|\\
    & \geq b^{-(t+1)} - \left|\sum_{i=t+2}^{k}(v_{i}-u_{i})b^{-i}\right| \\
    & \geq b^{-(t+1)}-(b^{-(t+1)}-b^{-k})\\
    & = b^{-k}.
\end{align*}
Hence $P(k)$ is $b^{-k}$-separated, and so it is natural to take $\beta(k)=b^{k}$. To show it is maximal consider any point $x\in\cC(b,J)$. Then there exists word $\bi \in J^{\N}$ such that
\begin{equation*}
    x=\lim_{n\to \infty}f_{(i_{1},\dots, i_{n})}(z).
\end{equation*}
Hence,
\begin{equation*}
    |f_{(i_{1}, \dots,i_{k})}(z)-x|=\left|f_{(i_{1},\dots,i_{k})}(z)-\lim_{n\to \infty}f_{(i_{1},\dots, i_{n})}(z) \right| \leq b^{-k}
\end{equation*}
and so $P(k)$ is maximal. Let $\cQ(,\cC(b,J),z)=\bigcup_{k\in\N}P(k)$. \par 
Let $S=\{k_{j}\}_{j\in\N}$ be a sequence of increasing integers. For $\tau \in \R_{+}$ define 
\begin{equation*}
\mathbf{W}^{S}_{\cC(b,J)}(\tau, z):=\left\{ x \in \cC(b,J) : |x-f_{\bi}(z)|<b^{-k_{j}\tau}\, \quad \bi \in J^{k_{j}}\, , \quad \text{ for all } j \in \N \right\}.
\end{equation*}
Note since $S$ is strictly increasing function $\beta(k)=b^{k}$ is increasing and unbounded on $S$. Furthermore $\mathbf{W}^{S}_{\cC(b,J)}(\tau, z)=\Lambda_{\cQ(\cC(b,J),z)}^{S}(\tau)$ and so by Theorem~\ref{main} we have the following.
\begin{theorem}
Let $S=\{k_{j}\}_{j\in\N}$ be a sequence of positive integers and let $h_{S}$ and $\alpha_{S}$ be defined by $S$ and satisfy the conditions as in Theorem~\ref{main}. For any $\tau\in\R_{+}$ such that $h_{S}>\tau>1$ we have that
\begin{equation*}
        \dimh \mathbf{W}^{S}_{\cC(b,J)}(\tau, z) = \frac{\gamma(b,J)}{\tau}\left(1 -\alpha_{S}(\tau-1) \right).
\end{equation*}    
\end{theorem}

\vspace{2ex}

\section{Proof of Theorem~\ref{main}}
Before giving the proof of Theorem~\ref{main} we show that $\Lambda_{\cQ}^{S}(\bt)$ can be written as a $\liminf$ sequence of rectangles. This $\liminf$ set will make for much easier calculation of the upper and lower bound of $\Lambda_{\cQ}^{S}(\bt)$. For $1\leq i \leq n$ and $j \in \N$, let 
\begin{equation*}
    E_{j,i}=\bigcup_{p_{i} \in P_{i}(q_{j})} B_{i}(p_{i},\beta(q_{j})^{-\tau_{i}}),
\end{equation*}
where $q_{j}$ denotes the $j$th value in the sequence $S$, and let
\begin{equation*}
    E_{j}=\prod_{i=1}^{n}E_{j,i}=\bigcup_{\bp=(p_{1},\dots, p_{n})\in P(q_{j})} \prod_{i=1}^{n}B_{i}(p_{i},\beta(q_{j})^{-\tau_{i}}).
\end{equation*}
Then
\begin{equation*}
    \Lambda_{\cQ}^{S}(\bt)=\bigcap_{j\in\N} E_{j}.
\end{equation*}

This construction is crucial in the proof of Theorem~\ref{main}. A brief sketch of the proof is as follows. As is standard with the calculation of the dimension of such sets, we split the proof into two parts, proving the upper and lower bound separately. \par
The upper bound is proven by considering the standard cover provided in the previous section. One small technicality is that the cover is a cover of rectangles, so in order to construct an efficient cover of $\Lambda_{\cQ}^{S}(\bt)$ we need to consider several different coverings of balls, depending on the side lengths of the rectangles in the layers $E_{j}$. \par
The methodology of the lower bound is as follows. Firstly we construct a Cantor set $L_{\infty}^{S}$ inside of $\Lambda_{\cQ}^{S}(\bt)$. This Cantor set is a very natural construction. Generally, starting at the layer $E_{1}$ as defined in the construction of $\Lambda^{S}_{\cQ}(\tau)$ we iteratively construct the Cantor set by simply including all rectangles from $E_{j}$ that are contained in some rectangle from the $E_{j-1}$ layer. Hence starting at $E_{1}$ we construct a nested Cantor set $L_{\infty}^{S}$ contained in $\Lambda^{S}_{\cQ}(\bt)$. For the second part, we then construct some measure $\nu$ on $L_{\infty}^{S}$. This measure is defined naturally such that the mass is evenly distributed over each rectangle appearing in the layer, and the sum of the mass of rectangles contained in a larger rectangle of the previous layer is equal to the mass of said larger rectangle. The calculation of the H\"{o}lder exponent of the measure $\mu$ for a general ball is quite technical since our constructed Cantor set is made of rectangles. To calculate the H\"{o}lder exponent we have to split into a number of cases determined by the size of the radius of our general ball relative to the side lengths of the rectangles in our Cantor set construction. Once the exponent is calculated we can apply the mass distribution principle to obtain our lower bound dimension result.\par
Before going into the proof we state the definitions of Hausdorff measure and dimension and state some easy, but essential, lemmas that will be required in the proof.

\subsection{Preliminaries}\label{prelims}
We begin by recalling the definition of Hausdorff measure and dimension, for a thorough exposition see \cite{F14}. Let $(F,d)$ be a metric space and $X \subset F$.
 Then for any $0 < \rho \leq \infty$, any finite or countable collection~$\{B_i\}$ of subsets of $F$ such that
$X\subset \bigcup_i B_i$ and 
\begin{equation*}
    r(B_i)=\inf\{r\geq 0: d(x,y)\leq r \quad ( x,y \in B_{i}) \}\leq \rho
\end{equation*}
is called a \emph{$\rho$-cover} of $X$.
Let
\[ 
\cH_{\rho}^{s}(X)=\inf \left\{\sum\limits_{i} r(B_i)^{s}\right\} \, ,
\]
where the infimum is taken over all possible $\rho$-covers 
$\{B_i\}$ of $X$. The \textit{$s$-dimensional Hausdorff measure of $X$} is defined to be
\[
\cH^s(X)=\lim_{\rho\to 0}\cH_\rho^s(X).
\]
For any set $X\subset F$ denote by $\dimh X$ the Hausdorff dimension of $X$, defined as
\[
\dimh X :=\inf\left\{s\geq 0\;:\; \cH^s(X)=0\right\}\,.
\]
A property that the Hausdorff dimension enjoys is that it is countably stable. That is, for a sequence of sets $X_{i}$ we have that
\begin{equation*}
    \dimh \bigcup_{i} X_{i} = \sup_{i} \dimh X_{i}\, .
\end{equation*}
We will use the following lemma to calculate the lower bound dimension result.
\begin{lemma}[Mass Distribution Principle]\label{Mass distribution principle}
Let $\nu$ be a Borel probability measure supported on a subset $X \subseteq F$. Suppose that for $s>0$ there exists constants $c,r_{0}>0$ such that
\begin{equation*} 
\mu(B) \leq c r(B)^{s}
\end{equation*}
for all open balls $B \subset F$ with $r(B)<r_{0}$. Then $\cH^{s}(X) \geq \frac{1}{c}$.
\end{lemma}
\par

We now state some easy lemmas in relation to our setup that will be used in both the upper and lower bound dimension calculations. The first lemma essentially gives us good bounds on the number of abstract rationals contained in some rectangle of certain sidlenghts. \par 
 
\begin{lemma} \label{vol_arg}
For any $x_{i}\in F_{i}$, $q\in \cN$ and $r>0$ such that $r>\beta(q)^{-1}$
\begin{equation*}
\frac{c_{1,i}}{c_{2,i}}2^{-\delta_{i}}(r\beta(q))^{\delta_{i}} \leq \# \left\{ p \in P_{i}(q): p \in B_{i}(x_{i},r)\right\} \leq \frac{c_{2,i}}{c_{1,i}}2^{\delta_{i}+1}(r\beta(q))^{\delta_{i}}.
\end{equation*}
\end{lemma}

\begin{proof}
Note that by the $\beta(q)^{-1}$ separated condition of $P_{i}(q)$ the balls 
\begin{equation*}
\bigcup_{p\in P_{i}(q)} B_{i}\left(p,\tfrac{1}{2}\beta(q)^{-1}\right)    
\end{equation*}
are disjoint, and so using a volume argument one can deduce that 
\begin{equation*}
   \# \left\{ p \in P_{i}(q): p \in B_{i}(x_{i},r)\right\} \leq \frac{\mu_{i}\left(B_{i}(x_{i},r)\right)}{\mu\left(B_{i}(p,\frac{1}{2}\beta(q)^{-1})\right)}+1 \leq \frac{c_{2,i}}{c_{1,i}}2^{\delta_{i}+1}(r\beta(q))^{\delta_{i}},
\end{equation*}
where the second inequality follows since $r\beta(q)>1$ and $c_{2,i}c_{1,i}^{-1}\geq 1$. \par 
A similar argument can be done for the lower bound. Note that the maximal condition of $P_{i}(q)$ ensures, regardless of the arrangement of $P_{i}(q)$, that at least one $p \in P_{i}(q)$ must be contained in any ball $B_{i}(x,2\beta(q)^{-1})$ with $x\in F_{i}$, and by a volume argument
\begin{equation*}
    \# \left\{ p \in P_{i}(q): p \in B_{i}(x_{i},r)\right\}\geq \frac{\mu_{i}\left(B_{i}(x_{i},r)\right)}{\mu_{i}\left(B_{i}(p,2\beta(q)^{-1})\right)} \geq \frac{c_{1,i}}{c_{2,i}}2^{-\delta_{i}}(r\beta(q))^{\delta_{i}}.
\end{equation*}
\end{proof}
The following lemma shows us that our sequence $S$, satisfying the conditions of Theorem~\ref{main}, will decrease fast enough.
\begin{lemma} \label{constant corrector}
    For any sequence $S=\{q_{j}\}_{j\in\N}$ if
    \begin{equation*}
        \inf_{j\in\N} \frac{\log \beta(q_{j})}{\log \beta(q_{j-1})}=h_{S}>1
    \end{equation*}
    then
    \begin{equation*}
        \lim_{j\to \infty} \frac{j}{\log \beta(q_{j})}=0. 
    \end{equation*}
\end{lemma}
\begin{proof}
    Observe the infimum condition on $S$ implies that for any $j\in\N_{>1}$
    \begin{equation*}
        \log \beta(q_{j}) >h_{S}\log \beta(q_{j-1})>h_{S}^{2}\log \beta(q_{j-2})> \dots > h_{S}^{j-1}\log \beta(q_{1}).
    \end{equation*}
    Thus
    \begin{equation*}
    \frac{j}{\log \beta(q_{j})} < \frac{j}{h_{S}^{j-1}\log \beta(q_{1})} \to 0
    \end{equation*}
    as $j \to \infty$ since $h_{S}>1$.
\end{proof}

%%%%%%%%%%%%%%%%%%%%%%%%%%%%%%%%%%%%%%
%
%   UPPER BOUND
%
%%%%%%%%%%%%%%%%%%%%%%%%%%%%%%%%%%%%%%

\subsection{Upper bound of Theorem~\ref{main}}
 
Observe that $E_{1}\cap\dots \cap E_{j}$ is a $\beta(q_{j})^{-\min_{i}\tau_{i}}$-cover of $\Lambda_{\cQ}^{S}(\bt)$ for any $j \in \N$. Consider the cover $E_{1}\cap E_{2}$ and note that for any $x_{i}\in F_{i}$
\begin{equation*}
    \#\{p_{i}\in P_{i}(q_{2}): p\in B_{i}(x_{i}, \beta(q_{1})^{-\tau_{i}})\} \overset{\text{Lemma~\ref{vol_arg}}}{\leq} \frac{c_{2,i}}{c_{1,i}}2^{\delta_{i}+1} (\beta(q_{2})\beta(q_{1})^{-\tau_{i}})^{\delta_{i}},
\end{equation*}
and so for any $\bx\in F$
\begin{equation*}
    \#\left\{\bp \in P(q_{2}): \bp \in \prod_{i=1}^{n} B_{i}(x_{i}, \beta(q_{1})^{-\tau_{i}})\right\} \overset{\text{Lemma~\ref{vol_arg}}}{\leq} \left(\prod_{i=1}^{n}\frac{c_{2,i}}{c_{1,i}}\right) 2^{n+\sum\limits_{i=1}^{n}\delta_{i}} \beta(q_{2})^{\sum\limits_{i=1}^{n}\delta_{i}} \beta(q_{1})^{-\sum\limits_{i=1}^{n}\tau_{i}\delta_{i}}.
\end{equation*}
We should remark here that Lemma~\ref{vol_arg} is applicable since $$\beta(q_{1})^{-\tau_{i}}>\beta(q_{1})^{-h_{S}}>\beta(q_{2})^{-1}$$ by our definition of $h_{S}$ and assumption on $\bt$. From now on, for ease of notation, we will write 
\begin{equation*}
\delta=\sum_{i=1}^{n}\delta_{i}\, , \quad c_{1}=\prod_{i=1}^{n}c_{1,i}\, , \quad  c_{2}=\prod_{i=1}^{n}c_{2,i}\, .
\end{equation*}
This tells us that the cover $E_{1}\cap E_{2}$ is composed of at most 
\begin{equation*}
    \frac{c_{2}}{c_{1}} 2^{n+\delta} \beta(q_{2})^{\delta}\beta(q_{1})^{-\sum\limits_{i=1}^{n}\tau_{i}\delta_{i}}
\end{equation*}
rectangles of the form
\begin{equation*}
    \prod_{i=1}^{n}B_{i}\left(p_{i},\beta(q_{2})^{-\tau_{i}}\right)
\end{equation*}
for $\bp=(p_{1},\dots,p_{n}) \in P(q_{2})$. We can repeat this process iteratively to obtain that $E_{1}\cap\dots\cap E_{j}$ can be covered by at least
\begin{equation*}
    G_{j}:=\#P(q_{1}) \left( \frac{c_{2}}{c_{1}} 2^{n+\delta}\right)^{j-1} \prod_{i=1}^{j-1}\beta(q_{i+1})^{\delta}\beta(q_{i})^{-\sum\limits_{i=1}^{n}\tau_{i}\delta_{i}}
\end{equation*}
rectangles of the form
\begin{equation*}
    \prod_{i=1}^{n}B_{i}\left(p_{i},\beta(q_{j})^{-\tau_{i}}\right)
\end{equation*}
for some $\bp \in P(q_{j})$. \par 
For now, fix some $1\leq k \leq n$. Observe that $\prod_{i=1}^{n}B_{i}\left(p_{i},\beta(q_{j})^{-\tau_{i}}\right)$ can be covered by
\begin{equation*}
    \asymp_{c_{1},c_{2}} \prod_{i=1}^{n} \max\left\{1, \beta(q_{j})^{(\tau_{k}-\tau_{i})\delta_{i}}\right\}
\end{equation*}
$n$-dimensional balls, in $F$, of radius $\beta(q_{j})^{-\tau_{k}}$. Note that the implied constant is independent of $j$. Hence
\begin{align} \label{upper_measure}
    \cH^{s}\left(\Lambda_{\cQ}^{S}(\bt)\right) &\ll_{c_{1},c_{2},} \liminf_{j\to \infty} G_{j}\prod_{i=1}^{n} \max\left\{1, \beta(q_{j})^{(\tau_{k}-\tau_{i})\delta_{i}}\right\} \left(\beta(q_{j})^{-\tau_{k}}\right)^{s} \nonumber\\
&\ll_{c_{1},c_{2},\#P(q_{1})}\liminf_{j\to \infty} \left( \frac{c_{2}}{c_{1}} 2^{n+\delta}\right)^{j-1} \prod_{i=1}^{j-1}\beta(q_{i+1})^{\delta}\beta(q_{i})^{-\sum\limits_{i=1}^{n}\tau_{i}\delta_{i}}\beta(q_{j})^{\sum\limits_{i:\tau_{k}>\tau_{i}}(\tau_{k}-\tau_{i})\delta_{i}} \left(\beta(q_{j})^{-\tau_{k}}\right)^{s}.
\end{align}
For
\begin{equation} \label{s_value}
    s>\frac{1}{\tau_{k}}\left(\delta + \frac{(j-1)\log\left( \frac{c_{2}}{c_{1}} 2^{n+\delta}\right)}{\log \beta(q_{j})}-\left(\sum\limits_{i=1}^{n}(\tau_{i}-1)\delta_{i}\right)\frac{\sum\limits_{i=1}^{j-1}\log \beta(q_{i})}{\log \beta(q_{j})} + \sum\limits_{i:\tau_{k}>\tau_{i}}(\tau_{k}-\tau_{i})\delta_{i} \right)
\end{equation}
 we have that \eqref{upper_measure} is finite. Furthermore, for any $\epsilon>0$ there exists subsequence $\{j_{t}\}_{t\in\N}$ and large $t_{0}\in\N$ such that for any
 \begin{equation*}
     s>\frac{1}{\tau_{k}}\left(\delta + \epsilon -\left(\sum_{i=1}^{n}(\tau_{i}-1)\delta_{i}\right)\alpha_{S} + \sum_{i:\tau_{k}>\tau_{i}}(\tau_{k}-\tau_{i})\delta_{i} \right),
 \end{equation*}
the equation after the limit in \eqref{upper_measure} is finite for all $j_{t}$ with $t>t_{0}$. This follows from the definition of $\alpha_{S}$, and Lemma~\ref{constant corrector}. Thus
\begin{equation*}
    \dimh \Lambda_{\cQ}^{S}(\bt) \leq  \frac{1}{\tau_{k}}\left(\delta -\left(\sum_{i=1}^{n}(\tau_{i}-1)\delta_{i}\right)\alpha_{S} + \sum_{i:\tau_{k}>\tau_{i}}(\tau_{k}-\tau_{i})\delta_{i} \right)\, .
\end{equation*}\par 
We can repeat the above steps for each $1\leq k\leq n$, and so
\begin{equation*}
    \dimh \Lambda_{\cQ}^{S}(\bt) \leq \min_{1\leq k \leq n} \left\{ \frac{1}{\tau_{k}}\left(\delta -\left(\sum_{i=1}^{n}(\tau_{i}-1)\delta_{i}\right)\alpha_{S} + \sum_{i:\tau_{k}>\tau_{i}}(\tau_{k}-\tau_{i})\delta_{i} \right) \right\}.
\end{equation*}

%%%%%%%%%%%%%%%%%%%%%%%%%%%%%%%%%%%%%%
%
%   LOWER BOUND
%
%%%%%%%%%%%%%%%%%%%%%%%%%%%%%%%%%%%%%%

\subsection{Lower bound of Theorem~\ref{main}}
 We begin with the construction of the Cantor subset $L_{\infty}^{S}$ of $\Lambda^{S}_{\cQ}(\bt)$ before constructing a measure $\nu$ with support $L_{\infty}^{S}$. Finally we calculate the H\"{o}lder exponent of a general ball for such measure. 
\subsubsection{Cantor construction of $L_{\infty}^{S}$}
The Cantor set $L_{\infty}^{S}$ is constructed as follows: \\
\begin{enumerate}
\item[Level 1:] Recalling that
\begin{equation*}
    \Lambda_{\cQ}^{S}(\bt)=\bigcap_{j\in\N} E_{j}
\end{equation*}
Let $L_{1}^{S}=E_{1}$. \\

\item[Level 2:] Let
\begin{equation*}
    L_{2}^{S}=\underset{\bp \in \prod_{i=1}^{n}B_{i}\left(p_{i},\frac{1}{2}\beta(q_{1})^{-\tau_{i}}\right)}{\bigcup_{\bp\in P(q_{2}):}} \prod_{i=1}^{n}B_{i}\left(p_{i},\beta(q_{2})^{-\tau_{i}}\right).
\end{equation*}
Observe that for any $\bx \in F$
\begin{equation*}
    \#\left\{\bp\in P(q_{2}): \bp \in \prod_{i=1}^{n}B_{i}\left(x_{i}, \frac{1}{2}\beta(q_{1})^{-\tau_{i}}\right) \right\} \overset{\text{Lemma~\ref{vol_arg}}}{\geq} \prod_{i=1}^{n}\frac{c_{1,i}}{c_{2,i}}2^{-2\delta_{i}}(\beta(q_{2})\beta(q_{1})^{-\tau_{i}})^{\delta_{i}} 
\end{equation*}
and so $L_{2}^{S}$ is composed of at least
\begin{equation} \label{L1-count}
    \frac{c_{1}}{c_{2}} 2^{-2\delta}\beta(q_{2})^{\delta}\beta(q_{1})^{-\sum_{i=1}^{n}\tau_{i}\delta_{i}}
\end{equation}
rectangles of the form
\begin{equation*}
    \prod_{i=1}^{n}B_{i}\left(p_{i},\beta(q_{2})^{-\tau_{i}}\right)
\end{equation*}
for some $\bp=(p_{1},\dots,p_{n}) \in P(q_{2})$. Note that \eqref{L1-count} may be small but since we are taking integer values there is at least $1$ rectangle.\\

\item[Level $L_{j}^{S}$:] The argument above can be applied inductively again. Namely, for rectangle $R_{j-1} \in L_{j-1}^{S}$\footnote{We should note here that technically $L_{j-1}^{S}$ is a union of rectangles, rather than a collection of rectangles so strictly "$R_{j-1}\in L_{j-1}^{S}$" does not make sense. Hopefully it is clear that in this notation we mean one of the rectangles that is in the union of rectangles in the construction of $L_{j-1}^{S}$.} let
\begin{equation*}
    L_{j}^{S}(R_{j-1})= \underset{\bp \in R_{j-1}}{\bigcup_{\bp\in P(q_{2}):}} \prod_{i=1}^{n}B_{i}(p_{i},\beta(q_{j})^{-\tau_{i}}),
\end{equation*}
and
\begin{equation*}
    L_{j}^{S}=\bigcup_{R_{j-1}\in L_{j-1}^{S}}L_{j}^{S}(R_{j-1}).
\end{equation*}
Observe that using the same calculation as in the construction of $L_{2}^{S}$ that
\begin{equation*}
    \#\{\bp \in P(q_{j}): \bp \in R_{j-1}\} \geq \frac{c_{1}}{c_{2}}2^{-2\delta}\beta(q_{j})^{\delta}\beta(q_{j-1})^{-\sum\limits_{i=1}^{n}\tau_{i}\delta_{i}}
\end{equation*}
and so each set $L_{j}^{S}(R_{j-1})$ is composed of at least
\begin{equation} \label{L-count}
    \frac{c_{1}}{c_{2}}2^{-2\delta}\beta(q_{j})^{\delta}\beta(q_{j-1})^{-\sum\limits_{i=1}^{n}\tau_{i}\delta_{i}}
\end{equation}
rectangles of the form
\begin{equation*}
    \prod_{i=1}^{n}B_{i}\left(p_{i},\beta(q_{j})^{-\tau_{i}}\right).
\end{equation*}
Note that, by our definition of $h_{S}$ and condition that $h_{S}>\tau_{i}$ for each $1\leq i \leq n$, for sufficiently large $j$ \eqref{L-count} will be strictly larger than $1$. Hence the constructed Cantor set does not form a singleton.
\end{enumerate}
To finish the construction define.
\begin{equation*}
    L_{\infty}^{S}=\bigcap_{j\in\N}L_{j}^{S}.
\end{equation*}

\subsubsection{Construction of measure $\nu$ on $L_{\infty}^{S}$}

Define the measure $\nu$ of $L_{\infty}^{S}$ by
\begin{equation*}
    \nu(R_{0})=1 \quad \text{ for } R_{0}=F,
\end{equation*}
and for each $R_{j}\in L_{j}^{S}(R_{j-1})$ define
\begin{equation} \label{measure def}
    \nu(R_{j})=\nu(R_{j-1})\frac{1}{\#L_{j}^{S}(R_{j-1})}.
\end{equation}
That is, we equally distribute the mass between rectangles in each layer. It is easy to see that the mass distribution is measure-preserving and so $\nu$ is defined as
\begin{equation*}
    \nu(A)=\inf\left\{\sum_{i}\nu(R_{i}) : \bigcup R_{i} \supseteq A \quad \text{ and } R_{i} \in \bigcup_{j\in\N}L_{\infty}^{S} \right\}
\end{equation*}
for any Borel set $A\subseteq F$ is a Borel probability measure with support $L^{S}_{\infty} \subset \Lambda^{S}_{\cQ}(\bt)$.\par 
Consider a rectangle $R_{k}\in L_{k}^{S}$. Observe that $$r(R_{k})=\max\limits_{1\leq i \leq n}\beta(q_{k})^{-\tau_{i}}=\beta(q_{k})^{-\min\limits_{1\leq i \leq n} \tau_{i}},$$ and by definition that
\begin{align}
    \nu(R_{k}) & \overset{\eqref{L-count}+\eqref{measure def}}{\leq} \nu(R_{k-1})\frac{c_{2}}{c_{1}}2^{2\delta}\beta(q_{k})^{-\delta}\beta(q_{k-1})^{\sum\limits_{i=1}^{n}\tau_{i}\delta_{i}}, \nonumber\\
    & \leq \beta(q_{1})^{-\delta} \prod_{j=2}^{k} \frac{c_{2}}{c_{1}}2^{2\delta}\beta(q_{j})^{-\delta}\beta(q_{j-1})^{\sum\limits_{i=1}^{n}\tau_{i}\delta_{i}}, \nonumber\\
    & \leq \beta(q_{k})^{-\delta}\prod_{j=1}^{k-1}\frac{c_{2}}{c_{1}}2^{2\delta}\beta(q_{j})^{-\sum\limits_{i=1}^{n}(1-\tau_{i})\delta_{i}}. \label{rectangle measure}
\end{align}
Thus the H\"{o}lder exponent can be calculated to be
\begin{align*}
    \frac{\log \nu(R_{k})}{\log r(R_{k})} & \geq \frac{\delta \log \beta(q_{k}) +(k-1)\log\left(\frac{c_{1}}{c_{2}}2^{-2\delta}\right) - \left(\sum\limits_{i=1}^{n}(\tau_{i}-1)\delta_{i} \right)\left(\sum\limits_{j=1}^{k-1}\log \beta(q_{j})\right)}{\min\limits_{1\leq i \leq n}\tau_{i}\log \beta(q_{k})}\\
    & = \frac{1}{\min\limits_{1\leq i \leq n}\tau_{i}}\left( \delta + \log\left(\frac{c_{1}}{c_{2}}2^{-2\delta}\right)\frac{(k-1)}{\log \beta(q_{k})}-\left(\sum\limits_{i=1}^{n}(\tau_{i}-1)\delta_{i} \right) \frac{\sum\limits_{j=1}^{k-1}\log \beta(q_{j})}{\log \beta(q_{k})} \right).
\end{align*}

By our condition on the sequence $S$ and Lemma~\ref{constant corrector} for any $\varepsilon>0$ there exists $k_{\varepsilon}$ such that for all $k>k_{\varepsilon}$
\begin{equation} \label{epsilon size 1}
    \left|\frac{\sum\limits_{j=1}^{k-1}\log \beta(q_{j})}{\log \beta(q_{k})}-\alpha_{S} \right|< \frac{\varepsilon}{2}\left(\sum\limits_{i=1}^{n}(\tau_{i}-1)\delta_{i}\right)^{-1} \, , \quad \text{ and } \quad \left|\log\left(\frac{c_{1}}{c_{2}}2^{-2\delta}\right)\frac{(k-1)}{\log \beta(q_{k})}\right|< \frac{\varepsilon}{2}\, .
\end{equation}
Hence we have that 
\begin{equation} \label{general s bound}
\frac{\log \nu(R_{k})}{\log r(R_{k})} \geq \frac{1}{\min\limits_{1\leq i \leq n}\tau_{i}}\left( \delta - \left(\sum\limits_{i=1}^{n}(\tau_{i}-1)\delta_{i} \right)\alpha_{S}+\varepsilon \right):=s_{\min},
\end{equation}
Thus
\begin{equation*}
    \nu(R_{k}) \ll r(R_{k})^{s_{\min}}.
\end{equation*}

\subsubsection{H\"{o}lder exponent of a general ball}
Fix any $\varepsilon>0$ and let $k_{\varepsilon} \in \N$ be large enough such that \eqref{general s bound} holds for all $k > k_{\varepsilon}$ and
\begin{equation} \label{epsilon size 2}
    \left| \frac{\log\left( \frac{c_{2}}{c_{1}}4^{n+\delta}\right) +(k-1)\log\left(\frac{c_{1}}{c_{2}}2^{\delta}\right)}{\log \beta(q_{k-1})} \right|<\frac{\varepsilon}{2} \, , \quad \text{ and } \quad \left|\frac{k\log\left(\frac{c_{2}}{c_{1}}2^{2\delta}\right)+\log 2^{n+3\delta}}{\log \beta(q_{k-1})}\right|<\frac{\varepsilon}{2}\, .
\end{equation}
Note such choice of $k_{\varepsilon}$ is possible by \eqref{constant corrector}. Consider an arbitrary ball $B(x,r) \subset F$ with $x\in L_{\infty}^{S}$ and $0<r<r_{0}:=\beta(q_{k_{\varepsilon}})^{-\min_{1\leq i \leq n} \tau_{i}}$. If $B(x,r)$ intersects exactly $1$ rectangle at each layer of $L_{\infty}^{S}$ then
    \begin{equation*}
        \nu(B(x,r)) \leq \nu(R_{k}) \to 0 \quad \text{ as } k \to \infty,
    \end{equation*}
and so trivially $\nu(B(x,r))\leq r^{s}$. Thus we may assume there exists some $k \in \N_{>k_{\varepsilon}}$ such that $B(x,r)$ intersects exactly one rectangle in $L_{k-1}^{S}$, say $R_{k-1}$, and at least two rectangles in $L_{k}^{S}(R_{k-1})$. Note that
\begin{equation*}
    \nu(B(x,r))=\nu(B(x,r)\cap R_{k-1}).
\end{equation*}
Consider the following cases:

\begin{enumerate}
    \item[i)] $r>\beta(q_{k-1})^{-\min\limits_{1\leq i \leq n}\tau_{i}}$: Then
    \begin{equation*}
        \nu(B(x,r)) \leq \nu(R_{k-1}) \leq \left(\beta(q_{k-1})^{-\min\limits_{1\leq i \leq n}\tau_{i}}\right)^{s_{\min}} \leq r^{s_{\min}}.
    \end{equation*}
        
    \item[ii)] $\beta(q_{k-1})^{-\min\limits_{1\leq i \leq n}\tau_{i}}\geq r \geq \beta(q_{k-1})^{-\max\limits_{1\leq i \leq n}\tau_{i}}$: We need to find an upper bound of
    \begin{equation*}
        \lambda=\#\left\{ \bp \in P(q_{k}): \bp \in \prod_{i=1}^{n} B_{i}\left(x_{i},\min\left\{ r, 2\beta(q_{k-1})^{-\tau_{i}}\right\}\right) \right\}.
    \end{equation*}
That is, the number of centers corresponding to rectangles from $L_{k}^{S}(R_{k-1})$ contained in $B(x,r)\cap R_{k-1}$. Observe that
\begin{align*}
    \lambda &\overset{\text{Lemma~\ref{vol_arg}}}{\leq} \prod_{i: r<2\beta(q_{k-1})^{-\tau_{i}}}\frac{c_{2,i}}{c_{1,i}}(r\beta(q_{k}))^{\delta_{i}} \times \prod_{i: r\geq 2\beta(q_{k-1})^{-\tau_{i}}} \frac{c_{2,i}}{c_{1,i}}4^{1+\delta_{i}}(\beta(q_{k})\beta(q_{k-1})^{-\tau_{i}})^{\delta_{i}}, \\
    & \leq \frac{c_{2}}{c_{1}}  4^{n+\delta} \beta(q_{k})^{\delta}\prod_{i: r<2\beta(q_{k-1})^{-\tau_{i}}}r^{\delta_{i}} \times \prod_{i: r\geq 2\beta(q_{k-1})^{-\tau_{i}}} \beta(q_{k-1})^{-\tau_{i}\delta_{i}}.
\end{align*}
Hence
\begin{align}
    \nu(B(x,r)\cap R_{k-1}) &\leq \underset{B(x,r)\cap R_{k} \neq \emptyset}{\sum_{R_{k}\in L_{k}^{S}(R_{k-1}):}} \nu(R_{k}) \nonumber\\[2ex]
    & \leq \lambda \, \nu(R_{k}), \nonumber \\
    & \overset{\eqref{rectangle measure}}{\leq} \lambda \, \beta(q_{k})^{-\delta}\prod_{j=1}^{k-1}\frac{c_{2}}{c_{1}}2^{2\delta}\beta(q_{j})^{-\sum\limits_{i=1}^{n}(1-\tau_{i})\delta_{i}}. \label{eq1}
\end{align}
Note that $B(x,r)\cap R_{k-1}$ is contained in a ball with radius $r$, and so for $$N_{\nu,r,k}=\frac{\log\nu(B(x,r)\cap R_{k-1})}{\log r(B(x,r)\cap R_{k-1})}$$
we have
\begin{align*}
     N_{\nu,r,k} & \geq \frac{\left(\sum\limits_{i:r<2\beta(q_{k-1})^{-\tau_{i}}}\delta_{i}\right)\log r - \left(\sum\limits_{i:r\geq 2\beta(q_{k-1})^{-\tau_{i}}}\tau_{i}\delta_{i}\right)\log \beta(q_{k-1}) -\left(\sum\limits_{i=1}^{n}(1-\tau_{i})\delta_{i}\right)\left(\sum\limits_{i=1}^{k-1}\log \beta(q_{j}) \right)}{\log r} \\
    & \hspace{2cm} + \frac{\log\left(\frac{c_{2}}{c_{1}}4^{n+\delta}\right) + (k-1)\log\left(\frac{c_{2}}{c_{1}}2^{2\delta}\right)}{\log r}
\end{align*}
Since $$\beta(q_{k-1})^{-\min\limits_{1\leq i \leq n}\tau_{i}}\geq r \geq \beta(q_{k-1})^{-\max\limits_{1\leq i \leq n}\tau_{i}},$$ there exists two coordinate axes, say the $u$th and $v$th axis, such that $r$ lies in the interval $[\beta(q_{k-1})^{-\tau_{u}},\beta(q_{k-1})^{-\tau_{v}}]$ and no $\beta(q_{k-1})^{-\tau_{i}}$ is contained in the interior. Observe the right-hand side of the above inequality is monotonic (in $r$ over $[\beta(q_{k-1})^{-\tau_{u}},\beta(q_{k-1})^{-\tau_{v}}]$) and so the minimum is obtained at one of the endpoints. Let 
\begin{align*}
    T_{1}&=\left\{i: \beta(q_{k-1})^{-\tau_{j}}<2\beta(q_{k-1})^{-\tau_{i}} \right\}\, , \\
    T_{2}&=\left\{ i: \beta(q_{k-1})^{-\tau_{j}}\geq 2\beta(q_{k-1})^{-\tau_{i}} \right\}=\{1,\dots, n\} \backslash T_{1}\, .  
\end{align*}
Then we may write
\begin{align*}
    N_{\nu,r,k} &\geq \min_{j=u,v} \left\{ \begin{array}{c} \frac{\left(\sum\limits_{i\in T_{1}}\delta_{i}\right)(-\tau_{j})\log \beta(q_{k-1}) - \left(\sum\limits_{i\in T_{2}}\tau_{i}\delta_{i}\right)\log \beta(q_{k-1}) -\left(\sum\limits_{i=1}^{n}(1-\tau_{i})\delta_{i}\right)\left(\sum\limits_{i=1}^{k-1}\log \beta(q_{j}) \right)}{-\tau_{j}\log \beta(q_{k-1})} \\[2ex]
     \hspace{2cm} +   \frac{\log\left(\frac{c_{2}}{c_{1}}4^{n+\delta}\right) + (k-1)\log\left( \frac{c_{1}}{c_{2}}2^{2\delta}\right)}{-\tau_{j}\log \beta(q_{k-1})} \end{array}\right\} \\
     & \geq \min_{j=u,v} \left\{ \frac{1}{\tau_{j}}\left(\begin{array}{c} \frac{\left(\sum\limits_{i\in T_{1}}\delta_{i}\right)(-\tau_{j})\log \beta(q_{k-1}) - \left(\sum\limits_{i\in T_{2}}\tau_{i}\delta_{i}\right)\log \beta(q_{k-1}) -\left(\sum\limits_{i=1}^{n}(1-\tau_{i})\delta_{i}\right)\left(\sum\limits_{i=1}^{k-1}\log \beta(q_{j}) \right)}{-\log \beta(q_{k-1})} \\[2ex]
     \hspace{2cm} +   \frac{\varepsilon}{2} \end{array}\right)\right\}\, ,
\end{align*}
by \eqref{epsilon size 2}. Now 
\begin{align*}
     N_{\nu,r,k} &\geq \min_{j=u,v} \left\{ \frac{1}{\tau_{j}}\left( \begin{array}{c}\left(\sum\limits_{i\in T_{1}}\delta_{i}\tau_{j}\right) + \left(\sum\limits_{i\in T_{2}}\tau_{i}\delta_{i}\right) 
      + \, \, \frac{\left(\sum\limits_{i=1}^{n}(1-\tau_{i})\delta_{i}\right)\left(\sum\limits_{i=1}^{k-1}\log \beta(q_{j}) \right)}{\log \beta(q_{k-1})}
     +\frac{\varepsilon}{2} \end{array} \right)\right\}\\
     &\geq \min_{j=u,v} \left\{ \frac{1}{\tau_{j}}\left( %\begin{array}{c} 
     \left(\sum\limits_{i\in T_{1}}\delta_{i}\tau_{j}\right) + \left(\sum\limits_{i\in T_{2}}\tau_{i}\delta_{i}\right) + \, \left(\sum\limits_{i=1}^{n}(1-\tau_{i})\delta_{i}\right)\left(1+ \frac{\left(\sum\limits_{i=1}^{k-2}\log \beta(q_{j}) \right)}{\log \beta(q_{k-1})}\right)
     +\frac{\varepsilon}{2} 
     %\end{array}
     \right)\right\}\\
     &\geq \min_{j=u,v} \left\{ \frac{1}{\tau_{j}}\left( \begin{array}{c} \left(\sum\limits_{i\in T_{1}}\delta_{i}\tau_{j}\right) + \left(\sum\limits_{i\in T_{2}}\tau_{i}\delta_{i}\right) 
     + \, \left(\sum\limits_{i=1}^{n}(1-\tau_{i})\delta_{i}\right)\left(1+ \alpha_{S}\right)
     +\varepsilon \end{array} \right)\right\}\, , \\
\end{align*}
by \eqref{epsilon size 1}. Splitting the third summation into the two components we get
\begin{align*}
     N_{\nu,r,k} &\geq \min_{j=u,v} \left\{ \frac{1}{\tau_{j}}\left( \begin{array}{c} \delta -\sum\limits_{i=1}^{n}\tau_{i}\delta_{i}+  \left(\sum\limits_{i\in T_{1}}\delta_{i}\tau_{j}\right) + \left(\sum\limits_{i\in T_{2}}\tau_{i}\delta_{i}\right) 
     + \, \left(\sum\limits_{i=1}^{n}(1-\tau_{i})\delta_{i}\right) \alpha_{S}
     +\varepsilon \end{array} \right)\right\}\\
     & \geq \min_{j=u,v} \left\{ \frac{1}{\tau_{j}}\left( \delta + \left(\sum_{i\in T_{1}}(\tau_{j}-\tau_{i})\delta_{i}\right)  +\left(\sum_{i=1}^{n}(1-\tau_{i})\delta_{i}\right)\alpha_{S}
     +\varepsilon \right)\right\} \\
      & \geq \min_{j=u,v} \left\{ \frac{1}{\tau_{j}}\left( \delta - \left(\sum_{i=1}^{n}(\tau_{i}-1)\delta_{i}\right)\alpha_{S} + \left(\sum_{i:\tau_{j}>\tau_{i}}(\tau_{j}-\tau_{i})\delta_{i}\right) + \varepsilon \right)\right\} =s_{u,v}\, .
\end{align*}
The last line follows on the observation that for $k$ sufficiently large and $j$ fixed, the sets $\{i:\tau_{j}>\tau_{i}\}$ and $T_{1}\cup \{j\}$ are the same. Clearly, the fact that $j$ is omitted does not affect the appearing summation.
This completes case ii).\\

    \item[iii)] $r \leq \beta(q_{k-1})^{-\max_{1\leq i \leq n} \tau_{i}}$: We calculate
    \begin{equation*}
        \nu\left(B(x,r) \cap R_{k-1}\right)\leq \underset{B(x,r)\cap R_{k} \neq \emptyset}{\sum_{R_{k}\in L_{k}^{S}(R_{k-1})}} \nu(R_{k}).
    \end{equation*}
    Since $B(x,r)$ intersects at least two rectangles in $L_{k}^{S}$, and $x$ is contained in one of them, we have that $r>\tfrac{1}{2}\beta(q_{k})^{-1}$. Hence
    \begin{align*}
        \#\left\{ R_{k}\in L_{k}^{S}: B(x,r)\cap R_{k} \neq \emptyset \right\} & \leq \#\left\{p \in \prod_{i=1}^{n}P_{i}(q_{k}):p\in\prod_{i=1}^{n}B_{i}(x_{i},4r)\right\} \\
        & \overset{\text{Lemma~\ref{vol_arg}}}{\leq} \left(\frac{c_{2}}{c_{1}}\right)^{n}2^{3\delta+n}(r\beta(q_{k}))^{\delta}.
    \end{align*}
    So
    \begin{align*}
        \nu(B(x,r) \cap R_{k-1})& \leq \frac{c_{2}}{c_{1}}2^{n+3\delta}(r\beta(q_{k}))^{\delta} \nu(R_{k}) \\
        & \leq \frac{c_{2}}{c_{1}}2^{n+3\delta}(r\beta(q_{k}))^{\delta} \prod_{i=1}^{k} \frac{1}{\#L_{i}^{S}(R_{i-1})} \\
        &\overset{\eqref{L-count}}{\leq} (r\beta(q_{k}))^{\delta}  \left(\frac{c_{2}}{c_{1}}\right)^{k+1}2^{(n+3\delta)+2k\delta} \beta(q_{1})^{-\delta}\prod_{i=2}^{k}\beta(q_{i})^{-\delta_{i}}\beta(q_{i-1})^{\sum\limits_{i=1}^{n}\tau_{i}\delta_{i}} \\
        &\leq 2^{n+3\delta}\left(\frac{c_{2}}{c_{1}}2^{(2\delta)}\right)^{(k+1)} r^{\delta}\prod_{j=1}^{k-1}\beta(q_{j})^{\sum\limits_{i=1}^{n}(\tau_{i}-1)\delta_{i}}.
    \end{align*}
    Hence the H\"{o}lder exponent can be calculated to be
    \begin{align*}
        \frac{\log \nu(B(x,r))}{\log r}&\geq \frac{\delta\log r + \left( \sum\limits_{i=1}^{n}(\tau_{i}-1)\delta_{i}\right)\left(\sum\limits_{j=1}^{k-1}\log \beta(q_{j})\right)}{\log r} + \frac{(k+1)}{\log r}\log\left(\frac{c_{2}}{c_{1}}2^{(2\delta)}\right)+\frac{\log 2^{n+3\delta}}{\log r}\\
        &\overset{\eqref{epsilon size 2}}{\geq} \delta + \frac{ \left( \sum\limits_{i=1}^{n}(\tau_{i}-1)\delta_{i}\right)\left(\sum\limits_{j=1}^{k-1}\log \beta(q_{j})\right)}{\log r}-\left(\max_{1\leq i \leq n}\tau_{i}\right)^{-1}\frac{\varepsilon}{2} \\
        & \hspace{-1cm}\overset{\hspace{0.3cm}\left(r<\beta(q_{k-1})^{-\max\limits_{1\leq i \leq n}\tau_{i}}\right)}{\geq} \frac{1}{\max\limits_{1\leq i \leq n}\tau_{i}} \left( \delta \max\limits_{1\leq i \leq n}\tau_{i} - \left(\sum\limits_{i=1}^{n}(\tau_{i}-1)\delta_{i}\right)\frac{\sum\limits_{j=1}^{k-1} \log \beta(q_{j})}{\log \beta(q_{k-1})} \right)-\left(\max_{1\leq i \leq n}\tau_{i}\right)^{-1}\frac{\varepsilon}{2} \\
        & \overset{\eqref{epsilon size 1}}{\geq} \frac{1}{\max_{1\leq i \leq n}\tau_{i}} \left( \delta \max\limits_{1\leq i \leq n}\tau_{i} - \left(\sum\limits_{i=1}^{n}(\tau_{i}-1)\delta_{i}\right)(\alpha_{S} +1)-\frac{\varepsilon}{2} \right) -\left(\max_{1\leq i \leq n}\tau_{i}\right)^{-1}\frac{\varepsilon}{2}\\
        &= \frac{1}{\max\limits_{1\leq i \leq n}\tau_{i}} \left( \delta  - (\alpha_{S}-\varepsilon)\sum\limits_{i=1}^{n}(\tau_{i}-1)\delta_{i} + \sum\limits_{i=1}^{n}\left( \left(\max\limits_{1\leq i \leq n} \tau_{i}\right) - \tau_{i}\right)\delta_{i}  -\varepsilon  \right) \\ 
        &=s_{\max}\, .
    \end{align*}
    This completes case iii).\\
\end{enumerate}
Compiling the results of i)-iii) we have that
\begin{equation*}
    \mu(B(x,r))\ll r^{\min\limits_{1\leq u,v\leq n}\left\{s_{\min},s_{\max},s_{u,v}\right\}}
\end{equation*}
Note $\varepsilon>0$ was arbitrary in the values of $s_{\min},s_{\max}$ and each $s_{u,v}$. Thus letting $\varepsilon\to 0$ and noting that
\begin{equation*}
    \min_{1\leq u,v\leq n}\left\{s_{\min},s_{\max},s_{u,v}\right\}=s
\end{equation*}
gives us
\begin{equation*}
    \dimh \Lambda_{\cQ}^{S}(\bt) \geq s
\end{equation*}
via Lemma~\ref{Mass distribution principle} hence completing the proof of the lower bound.

%%%%%%%%%%%%%%%%%%%%%%%%%%%%%%%%%%%%%%
%
%   PROOF OF LIMINF SET
%
%%%%%%%%%%%%%%%%%%%%%%%%%%%%%%%%%%%%%%

\subsection{Proof of Theorem~\ref{corollary}} \label{Section:CorollaryProof}
To prove Theorem~\ref{corollary} via Theorem~\ref{main} note that by the countable stability of the Hausdorff dimension
\begin{equation*}
    \dimh \widehat{\Lambda}_{\cQ}^{S}(\bt) = \sup_{t\in\N}\dimh \Lambda_{\cQ}^{\sigma^{t}S}(\bt).
\end{equation*}
Observe that $h\geq h_{\sigma^{t}S}$ for any $t\in\N$. For $\bt$ chosen in Theorem~\ref{corollary} it may be true that there exists $\tau_{i}$ with $\tau_{i}>h_{\sigma^{t}S}$. However, these sets have dimension less than or equal to the exact dimension result of Theorem~\ref{main} so can be ignored when considering the supremum over $t\in\N$. Thus without loss of generality assume that for each $1\leq i \leq n$ $h>\tau_{i}>1$. Furthermore choose $t_{0}$ sufficiently large such that 
\begin{equation*}
    h>h_{\sigma^{t}S}>\tau_{i} \quad \forall 1\leq i \leq n
\end{equation*}
for all $t\geq t_{0}$. Thus by Theorem~\ref{main}
\begin{equation*}
    \dimh \widehat{\Lambda}_{\cQ}^{S}(\bt) = \sup_{t\in\N} s =s
\end{equation*}
completeing the proof.

%%%%%%%%%%%%%%%%%%%%%%%%%%%%%%%%%%%%%%
%
%   PROOF OF DIRICHLET SET
%
%%%%%%%%%%%%%%%%%%%%%%%%%%%%%%%%%%%%%%

\subsection{Proof of Theorem~\ref{interesting?}}

Note that if
\begin{equation*}
    \lim_{j\to \infty} \frac{\log q_{j}}{\log q_{j-1}}=k
\end{equation*}
then for any sufficiently small $\varepsilon>0$ there exists $j_{0}\in \N$ such that for all $j>j_{0}$
\begin{equation*}
    (k-\varepsilon)\log q_{j-1}<\log q_{j}< (k+\varepsilon)\log q_{j-1}.
\end{equation*}
So
\begin{equation*}
    \frac{\sum\limits_{i=1}^{j-1}\log q_{i}}{\log q_{j}}< \sum\limits_{i=1}^{j-1}(k+\varepsilon)^{-i}=\frac{1-(k+\varepsilon)^{-(j-1)}}{k-1+\varepsilon},
\end{equation*}
and similarly for the lower bound
\begin{equation*}
    \frac{\sum\limits_{i=1}^{j-1}\log q_{i}}{\log q_{j}}>
\frac{1-(k-\varepsilon)^{-(j-1)}}{k-1-\varepsilon}.
\end{equation*}
Thus, as $k>1$ and $0<\varepsilon$ can be chosen arbitrarily small ($\varepsilon<k-1$), taking the limit as $j \to \infty$ gives us that
\begin{equation*}
    \alpha_{S}=\lim_{j\to \infty} \frac{\sum\limits_{i=1}^{j-1}\log q_{i}}{\log q_{j}} =\frac{1}{k-1}.
\end{equation*}
Applying Corollary~\ref{real_corollary} with $h_{S}=k$, $\alpha_{S}=\frac{1}{k-1}$, and $0<\tau_{i}<h_{S}-1$ for each $1\leq i \leq n$ gives us Theorem~\ref{interesting?}.
\bibliographystyle{plain}
\bibliography{biblioo}
\end{document}